\documentclass[final,leqno]{siamltex1213}
\usepackage{amsmath}
\usepackage{amsbsy}
\usepackage{amsfonts}
\usepackage{algorithm}
\usepackage{algorithmic}
\newcommand{\rmnum}[1]{\romannumeral #1}

\title{A Transformation Approach that Makes SPAI, PSAI and RSAI Procedures
Efficient for Large Double Irregular Nonsymmetric Sparse Linear Systems\thanks
{Supported in part by National Science Foundation of China (No. 11371219).}}
\author{Zhongxiao Jia\thanks{Department of Mathematical
Sciences, Tsinghua University, Beijing, 100084,
People's Republic of China,
\email{jiazx@tsinghua.edu.cn}.}
\and Wenjie Kang\thanks{Department of Mathematical Sciences,
Tsinghua University, Beijing, 100084,
People's Republic of China, \email{kangwj11@mails.tsinghua.edu.cn}.}}

\begin{document}
\maketitle
\slugger{simax}{xxxx}{xx}{x}{x--x}

\begin{abstract}
A sparse matrix is called double
irregular sparse if it has at least one relatively dense column and row,
and it is double regular sparse if all the columns and rows of it are sparse.
The sparse approximate inverse preconditioning procedures
SPAI, PSAI($tol$) and RSAI($tol$) are costly and even impractical
to construct preconditioners
for a large sparse nonsymmetric linear system with the coefficient matrix
being double irregular sparse, but they are efficient for double regular sparse
problems.
Double irregular sparse linear systems have a wide range of applications, and
24.4\% of the nonsymmetric matrices in the Florida University collection are
double irregular sparse. For this class of problems, we propose
a transformation approach, which consists of four steps:
(i) transform a given double irregular sparse
problem into a small number of double regular sparse ones with the same
coefficient matrix $\hat{A}$, (ii) use SPAI, PSAI($tol$)
and RSAI($tol$) to construct sparse approximate inverses $M$ of $\hat{A}$, (iii)
solve the preconditioned double regular sparse linear systems
by Krylov solvers, and (iv) recover an approximate solution of
the original problem with a prescribed accuracy from those of the double
regular sparse ones.
A number of theoretical and practical issues are considered on
the transformation approach. Numerical experiments on a number of real-world
problems confirm the very sharp
superiority of the transformation approach to the standard approach that
preconditions the original double irregular sparse problem by SPAI, PSAI($tol$)
or RSAI($tol$) and solves the resulting preconditioned system by Krylov solvers.
\end{abstract}



\pagestyle{myheadings}
\thispagestyle{plain}
\markboth{ZHONGXIAO JIA AND WENJIE KANG}{Sparse Approximate Inverse
Preconditioning for Double Irregular Sparse S Problems}

\section{Introduction}\label{sec:1}
In scientific and engineering computing, a core task is to solve
the large sparse linear system
\begin{equation}\label{equation}
  Ax=b,
\end{equation}
where $A$ is an $n\times n$ real nonsymmetric and nonsingular matrix, and $b$
is a given $n$-dimensional vector. Krylov iterative solvers, such
as the generalized minimal residual method (GMRES) and the
biconjugate gradient stabilized method (BiCGStab) \cite{saad03},
have been commonly used for solving \eqref{equation} in nowdays.
However, the convergence of Krylov solvers is generally
extremely slow when $A$ is ill conditioned or has bad spectral
property \cite{saad03}. So it is necessary to use
preconditioning techniques to accelerate the convergence
of Krylov solvers. Sparse approximate inverse preconditioning procedures
have been one class of the most important general-purpose preconditioning
procedures over the past two decades \cite{benzi02,Ferronato12,saad03}. Their goal
is to construct a sparse approximate inverse $M\approx A^{-1}$ or
a factorized $M=M_1M_2\approx A^{-1}$ directly. The preconditioned
linear system is $AMy=b$ with $x=My$, $MAx=Ab$ or $M_1AM_2y=M_1b$ with $x=M_2y$,
which corresponds to the right, left or factorized preconditioning, respectively.

There are mainly two kinds of constructing a sparse approximate inverse
$M$. Typical algorithms of constructing a factorized sparse approximate inverse
$M=M_1M_2$ are approximate inverse (AINV) type algorithms
\cite{benzi96,benzi98} and the balanced incomplete factorization
(BIF) algorithm \cite{tuma08,tuma10}. Stabilized and block versions of the AINV
factorized approximate inverse preconditioner are proposed in \cite{benzi01}.
Kolotilina and Yeremin \cite{kolotilina93} have proposed a
factorized sparse approximate inverse (FSAI) preconditioning procedure
with a prescribed sparsity pattern of $M$.
FSAI has been generalized to block form, called BFSAI, in \cite{Janna10}.
An adaptive algorithm which gets the sparsity pattern of the BFSAI
preconditioner $M$ can be found in \cite{Janna14,Janna11,Janna13}.

The other kind of approach is based on F-norm minimization, which is
inherently parallelizable and computes a right preconditioner $M$ by
minimizing $\|AM-I\|_{F}$ with certain sparsity constraints
on $M$, where $\|\cdot\|_{F}$ denotes the Frobenius norm of a matrix
and $I$ is the identity matrix of order $n$. Applying this kind of approach to
$A^T$, one can compute a left preconditioner $M$.
A key of this kind of approach is the efficient determination of
an effective sparsity pattern of $A^{-1}$.
If the sparsity pattern of $M$ is prescribed, the resulting procedure
is called a static one; if the sparsity pattern of $M$ is adaptively
determined during the computational process, the procedure is called adaptive.
For a-priori effective sparsity patterns of $A^{-1}$, we refer the
reader to \cite{benson82,benson84,chow00,gilbert94,Huckle99}.
The SPAI algorithm proposed by Grote and Huckle \cite{Grote97}
has been a popular adaptive F-norm minimization based sparse approximate inverse
preconditioning procedure. It has been generalized
to block form, called BSPAI, in \cite{barnard99}.
Jia and Zhu \cite{Jia09} have
proposed an adaptive Power sparse approximate inverse  (PSAI) preconditioning
procedure and developed a practical PSAI($tol$) algorithm that,
during the loops, drops the nonzero entries
in $M$ whose sizes are below some tolerance $tol$. PSAI($tol$) has been shown
to be at least competitive with and can be substantially more effective than
SPAI \cite{Jia13b,Jia09}. Jia and Zhang \cite{Jia13a} have recently established
a mathematical theory on dropping tolerances $tol$ for all static F-norm
minimization based sparse approximate inverse procedures and
PSAI($tol$). Very recently,
the authors of this paper have proposed
a Residual based Sparse Approximate Inverse (RSAI) preconditioning procedure
\cite{Jia15}, which is different from the way used in SPAI and is based on only
the {\em partial but dominant} other than {\em all indices} of nonzero entries in
the current residual. They have developed
a practical RSAI($tol$) algorithm with dropping strategies exploited.
RSAI($tol$) improves the computational efficiency of SPAI substantially and
meanwhile constructs effective preconditioners $M$.
For more on sparse approximate inverse preconditioning procedures, we refer
the reader to \cite{benzi02,benzi99,chen05,Ferronato12,saad03}.

For SPAI and PSAI($tol$), Jia and Zhang \cite{Jia13b} have
investigated the efficiency of constructing $M$ and the preconditioning
effectiveness of $M$. They introduce the term of 'irregular sparse matrix', where
an irregular sparse $A$ means that it has at least one relatively dense column,
whose number of nonzero entries is substantially more than the average number of
nonzero entries per column of $A$. In implementations,
we call $A$ column irregular sparse if it has one column whose number of
nonzero entries is at least $10p$, where $p$ is the average number of nonzero
entries per column of $A$ \cite{Jia13b}.
Following this standard, it is reported in \cite{Jia13b}
that column irregular linear problems have a wide range of applications
and 34\% of the square matrices in the
University of Florida sparse matrix collection \cite{davis11}
are column irregular sparse; see \cite{Jia13b} for further information
on where column irregular sparse matrices come from and how dense irregular
columns are, etc. For a column irregular sparse $A$,
Jia and Zhang \cite{Jia13b} have shown that SPAI and PSAI($tol$) are
costly and may be impractical; they have given theoretical
arguments and numerical evidence
that $M$ obtained by SPAI may be ineffective for preconditioning \eqref{equation},
but $M$ by PSAI($tol$) is effective though its construction is costly.
Their analysis has also revealed that
SPAI and PSAI($tol$) are costly when applied to $A^T$ for computing
left preconditioners for $A$ column irregular sparse, that is,
we compute $M$ by minimizing $\|A^TM^T-I\|_F$
with certain sparsity constraints on $M$.

In the same way, we call $A$ row irregular sparse if it
has at least one relatively dense row. If $A$ is both column and row
irregular sparse, it is called double irregular
sparse. In contrast, if all the columns and rows of $A$ are sparse, $A$
is called double regular sparse. Using the same standard as
that of an irregular column, we define an irregular row. By this definition,
a column irregular symmetric matrix is double irregular sparse.
We have investigated all the real nonsymmetric square matrices in
the collection \cite{davis11}, which contains 775 matrices. We have found that
189 of them are double irregular sparse,
that is, 24.4\% of the nonsymmetric matrices in the collection are
double irregular. This indicates that double irregular sparse
linear systems have a wide range of practical applications.
Numerical experiments have illustrated that RSAI($tol$) improves the
computational efficiency of SPAI
substantially for column irregular sparse problems \cite{Jia15}.
However, we will see that RSAI($tol$) is expensive
and may be impractical for  row irregular sparse problems. This is also the
case for PSAI($tol$); see \cite{Jia13b}.

Summarizing the above, we come to the conclusion that SPAI, PSAI($tol$) and
RSAI($tol$) are costly and even impractical for double
irregular sparse problems. Therefore,
how to efficiently use the three preconditioning procedures
to solve double irregular sparse linear systems is of great importance. We
will focus this topic in the current paper.

As is known from \cite{Jia15,Jia13b,Jia09}, a
common and attractive feature of the aforementioned three procedures is that
they can construct preconditioners
$M$ efficiently for $A$ double regular sparse, among of
which PSAI($tol$) is most effective and SPAI and RSAI($tol$)
are comparably effective for preconditioning double regular sparse linear
systems. For the column irregular sparse \eqref{equation},
making use of the Sherman-Morrison-Woodbury formula,
Jia and Zhang \cite{Jia13b} have proposed an approach that
transforms \eqref{equation}
into a small number of column regular sparse ones with the same coefficient
matrix and multiple right-hand sides, so that SPAI and PSAI($tol$)
can construct preconditioners for
the regular sparse problems much more efficiently than they do
for \eqref{equation} directly.
An approximate solution of the original
system with a prescribed accuracy $\varepsilon$ is then recovered from those of
the regular sparse ones with the accuracy determined by $\varepsilon$.
The numerical experiments in \cite{Jia13b} have
indicated that such transformation approach
speeds up the computational efficiency of SPAI and PSAI($tol$) very substantially,
compared to them applied to the original problem. However,
the transformation approach does not suit for row irregular
sparse problems, for which SPAI, PSAI($tol$) and RSAI($tol$)
are still costly, as has been pointed out above.

In this paper, we show that a double irregular sparse
matrix $A$ can be expressed as the sum of a double regular sparse $\hat{A}$
and two certain low rank matrices. Motivated by the work \cite{Jia13b},
by exploiting the Sherman-Morrison-Woodbury formula \cite{matrixcomputations}
twice, we propose an approach that transforms the irregular sparse $A$ into
a regular sparse $\hat{A}$ and \eqref{equation} into $s_1+s_2+1$
linear systems with the same coefficient matrix $\hat{A}$ and $s_1+s_2+1$
right-hand sides, where $s_1$ and $s_2$ are the numbers of relatively dense
columns and rows, respectively. The transformation consists of two steps:
first transform $A$ with $s_1$ dense columns into a column regular
sparse matrix $\tilde{A}$, then transform $\tilde{A}$
with $s_2$ dense rows into a double regular sparse matrix $\hat{A}$.
Since $A$ is supposed to be sparse, $s_1$ and $s_2$
must be very small. As it will turn out,
the double regular sparse $\hat{A}$ equals $A$ minus two matrices of low ranks
$s_1$ and $s_2$, respectively.
Then we use SPAI, PSAI($tol$) and RSAI($tol$) to construct preconditioners
for these double regular sparse systems and solve the resulting preconditioned
linear systems by Krylov solvers. Finally, we
recover the solution of \eqref{equation} from those of the double regular sparse
ones. The above whole process is called the transformation approach.
We consider a number of theoretical and practical issues, including the
non-singularity of $\hat{A}$ and its conditioning. Particularly,
we prove how to design stopping criteria for the $s_1+s_2+1$ double regular linear
systems in order to ultimately obtain an approximate solution of \eqref{equation}
with a prescribed accuracy $\varepsilon$ from those of the $s_1+s_2+1$ double
regular ones. Numerical experiments
will exhibit the very sharp efficiency
of our transformation approach to the standard approach that first
preconditions \eqref{equation} by SPAI, PSAI($tol$) or RSAI($tol$)
and then solves the preconditioned linear system by Krylov solvers.
We will demonstrate that, due to the memory storage and huge
computational cost, SPAI, PSAI($tol$) and RSAI($tol$) are either out
of memory or cannot generate
preconditioners within 100 hours when directly applied to six of
the ten real-world double irregular sparse problems, but our transformation
approach works very efficiently for all the test problems and consumes
only a few seconds to no more than half an hour for the six hard problems
that the standard approach fails to solve.

The paper is organized as follows. In Section \ref{sec:2},
we briefly review SPAI, PSAI($tol$) and RSAI($tol$) procedures.
In Section \ref{sec:3}, we propose our transformation approach
for solving the double irregular sparse problem \eqref{equation}.
In Section \ref{sec:4}, we consider some theoretical and
practical issues. In Section~\ref{sec:5}, we report on numerical experiments,
confirming the very sharp superiority of our transformation approach to the
standard approach that
preconditions \eqref{equation} by SPAI, PSAI($tol$) or RSAI($tol$) directly
and solves it by Krylov solvers. Finally, we conclude the paper in
Section~\ref{sec:6}.

\section{The SPAI, PSAI($tol$) and RSAI($tol$) procedures}\label{sec:2}
For an F-norm minimization based sparse approximate inverse  preconditioning procedure, we
need to solve the constrained minimization problem
\begin{equation}\label{minfnorm}
\min_{M\in \mathcal{M}}\|AM-I\|_{F},
\end{equation}
where $\mathcal{M}$ is the set of matrices with a given sparsity
pattern $\mathcal{J}$. Define $\mathcal{M}_{k}$ as the set of
$n$-dimensional vectors whose sparsity pattern is
$\mathcal{J}_{k}=\{i\mid(i,k)\in \mathcal{J}\}$, and let
$M=(m_1,m_2,\ldots,m_n)$. Then \eqref{minfnorm} is recast as the $n$ independent
constrained least squares (LS) problems
\begin{equation}\label{min2norm}
  \min_{m_{k}\in \mathcal{M}_{k}}\|Am_{k}-e_{k}\|,\ k=1,2,\ldots,n,
\end{equation}
where $e_{k}$ is the $k$th column of $I$. Here and hereafter,
$\|\cdot\|$ denotes the 2-norm of a matrix or vector.
For each $k$, denote by $\mathcal{I}_{k}$
the set of indices of nonzero rows of $A(\cdot,\mathcal{J}_{k})$. Then
\eqref{min2norm} amounts to solving the smaller unconstrained LS problems
\begin{equation}\label{reducemin2norm}
\min_{\tilde{m}_{k}}\|A(\mathcal{I}_{k},\mathcal{J}_{k})m_{k}(\mathcal{J}_{k})
-e_{k}(\mathcal{I}_{k})\|,
k=1,2,\ldots,n,
\end{equation}
which can be solved by QR decompositions in parallel.

If $M$ is not yet good enough, that is,
\eqref{min2norm} does not drop below a prescribed tolerance $\eta$
for at least one $k$, one can use some adaptive sparse approximate inverse
preconditioning procedure, e.g., SPAI, PSAI($tol$) or
RSAI($tol$) to improve it by augmenting
or adjusting the sparsity pattern $\mathcal{J}_{k}$ dynamically.
We highlight that, mathematically, the unique fundamental
distinction of all the adaptive F-norm minimization based sparse approximate
inverse preconditioning procedures is the way that augments or adjusts
the sparsity pattern of $M$. It has been shown in \cite{Jia13b,Jia09}
that
PSAI($tol$) captures the sparsity pattern of $A^{-1}$ more effectively
than SPAI; moreover, the effectiveness of PSAI($tol$) is independent of
whether $A$ is (column or row) regular sparse or not, while SPAI is
more effective for regular sparse matrices than for irregular sparse ones.
In \cite{Jia15}, the authors have shown that
RSAI($tol$) and SPAI compute comparably effective preconditioners but
the former is more efficient than the latter.
In what follows we briefly review SPAI, PSAI($tol$) and RSAI($tol$).

\subsection{The SPAI procedure}\label{sec:2.1}
Denote by $\mathcal{J}_{k}^{(l)}$ the sparsity pattern of $m_{k}$
after $l$ loops starting with an initial pattern $\mathcal{J}_{k}^{(0)}$,
and define $\mathcal{I}_{k}^{(l)}$ to be
the set of all nonzero row indices of $A(\cdot,\mathcal{J}_{k}^{(l)})$.
Denote the residual of \eqref{min2norm} by
\begin{equation}\label{residual}
  r_{k}=Am_k-e_k.
\end{equation}
If $\|r_{k}\|=\|Am-e_k\|>\eta$, denote by $\mathcal{L}_{k}$ the set of indices $i$
for which $r_{k}(i)\neq 0$ and $\mathcal{N}_{k}$ the set
of indices of nonzero columns of $A(\mathcal{L}_{k},\cdot)$. Then
\begin{equation}\label{argument}
  \hat{\mathcal{J}}_{k}=\mathcal{N}_{k}\setminus\mathcal{J}_{k}^{(l)}
\end{equation}
forms the new candidates for augmenting $\mathcal{J}_{k}^{(l)}$
in the next loop of SPAI, in which $\mathcal{J}_{k}^{(l)}$
is updated as follows  \cite{Grote97}:
For each $j\in \hat{\mathcal{J}}_{k}$, consider the
one-dimensional minimization problem
\begin{equation}\label{onedimensional}
  \min_{\mu}\|r_{k}+\mu Ae_{j}\|,
\end{equation}
whose solution is
\begin{equation}\label{mu}
  \mu_{j}=-\frac{r_{k}^{T}Ae_{j}}{\|Ae_{j}\|^{2}},
\end{equation}
and the 2-norm $\rho_{j}$ of the new residual $r_{k}+\mu_{j}Ae_{j}$
satisfies
\begin{equation}\label{2normofresidual}
  \rho_{j}^{2}=\|r_{k}\|^{2}-\frac{(r_{k}^{T}Ae_{j})^{2}}
  {\|Ae_{j}\|^{2}}.
\end{equation}
SPAI selects a few, say $1\sim5$, most profitable indices from
$\hat{\mathcal{J}}_{k}$
with the smallest $\rho_{j}$ and adds them to $\mathcal{J}_{k}^{(l)}$
to obtain $\mathcal{J}_{k}^{(l+1)}$.
Define $\hat{\mathcal{I}}_{k}$ to be the set of indices
of new nonzero rows corresponding to the most profitable indices added, and
let $\mathcal{I}_{k}^{(l+1)}=\mathcal{I}_{k}^{(l)}
\bigcup \hat{\mathcal{I}}_{k}$. Then we solve the new LS problem
\begin{equation}\label{updateLS}
  \min\|A(\mathcal{I}_{k}^{(l+1)},\mathcal{J}_{k}^{(l+1)})
  m_k(\mathcal{J}_{k}^{(l+1)})
  -e_{k}(\mathcal{I}_{k}^{(l+1)})\|
\end{equation}
whose solution can be updated from the previous $m_{k}$ efficiently.
Proceed in such a way until $\|r_k\|\leq\eta$
or $l$ reaches the prescribed maximum $l_{\max}$, where
$\eta$ is a mildly small tolerance,
usually, $0.1\sim0.4$; see \cite{Grote97,Jia13b,Jia09}. Obviously, if the
cardinality of $\hat{\mathcal{J}}_{k}$ is very big,
SPAI is costly, which corresponds to the case that the $k$th column of $A$
is dense \cite{Jia13b,Jia09}. Precisely,
suppose that the $k$th column of $A$ is almost fully dense. SPAI has to
compute almost $n$ numbers $\rho_j$ by \eqref{2normofresidual}, then
sort almost $n$ indices in
$\hat{\mathcal{J}}_{k}$ by comparing the sizes of
$\rho_j$ and finally pick up a few most profitable indices
among them. This is very time consuming and can make SPAI fatally slow.
Similarly, it is easy to justify that
SPAI is costly when an index in $\mathcal{L}_k$ corresponds to a relatively
dense row of $A$, which results in a very big cardinality of $\hat{\mathcal{J}}_{k}$.
Therefore, SPAI is costly and even impractical for $A$ double irregular sparse.

\subsection{The PSAI($tol$) procedure}\label{sec:2.2}

We first review the basic PSAI (BPSAI) procedure proposed in \cite{Jia09}, which
is motivated by the Cayley--Hamilton theorem: $A^{-1}$ can be expressed as
\begin{equation}\label{cayleyHamilton}
 A^{-1}=\sum_{i=0}^{n-1}c_{i}A^{i},
\end{equation}
where $A^{0}=I$ and the $c_{i}$
are certain constants for $i=1,2,3,\ldots,n-1$. Denote by $\mathcal{P}(\cdot)$
the sparsity pattern of a matrix or vector, and define the matrix
$|A|=(|a_{ij}|)$. We see from \eqref{cayleyHamilton} that
$\mathcal{P}(A^{-1})\subseteq\mathcal{P}((I+|A|)^{n-1})$.
So the pattern $\mathcal{J}$ of a good sparse approximate inverse $M$ can be
taken as $\mathcal{P}((I+|A|)^{l_{\max}})$
for a given small $l_{\max}$ in BPSAI. As a result, the sparsity
pattern $\mathcal{J}_{k}$ of the $k$th column of $M$ is a subset of
$\bigcup_{i=0}^{l_{\max}}\mathcal{P}(|A|^{i}e_{k})$ since
$\mathcal{P}((I+|A|)^{l_{\max}})\subseteq\bigcup_{i=0}^{l_{\max}}\mathcal{P}(|A|^{i})$.

For $k=1,2,\ldots,n$, BPSAI updates the sparsity pattern $\mathcal{J}_{k}$
adaptively in the following way: For
$l=0,1,\ldots,l_{\max}$, denote by $\mathcal{J}_{k}^{(l)}$ the sparsity
pattern of $m_{k}$ at loop $l$ and by $\mathcal{I}_{k}^{(l)}$
the set of nonzero row indices of $A(\cdot,\mathcal{J}_{k}^{(l)})$. Define
$a_{k}^{(l+1)}=Aa_{k}^{(l)}$ with $a_{k}^{(0)}=e_{k}$. Then the sparsity
pattern
$\mathcal{J}_{k}^{(l+1)}=\mathcal{J}_{k}^{(l)}\bigcup\mathcal{P}(a_{k}^{(l+1)})$.
Denote by $\hat{\mathcal{I}}_{k}$ the set of
new nonzero row indices of $A(\cdot,\mathcal{J}_{k}^{(l+1)})$, and
let $\mathcal{I}_{k}^{(l+1)}=\mathcal{I}_{k}^{(l)}
\bigcup \hat{\mathcal{I}}_{k}$.
Then the new LS problem \eqref{reducemin2norm} can be solved
by updating $m_{k}$ instead of resolving it. Proceed in such a way
until $\|r_k\|\leq\eta$ or $l>l_{\max}$.

In order to develop a practical algorithm, we must control the
sparsity of $M$. Jia and Zhu \cite{Jia09} have proposed the PSAI($tol$)
algorithm, which, during the loops, drops those nonzero entries whose magnitudes
are below a prescribed threshold $tol$ and retains only those large ones.
It has turned out that the choice of $tol$ has strong effects on the
preconditioning effectiveness
and sparsity of $M$. Jia and Zhang \cite{Jia13a} have established a mathematical
theory on robust dropping tolerances for PSAI($tol$) and all the static
F-norm minimization based sparse approximate inverse
preconditioning procedures. According to the theory, they have
designed an adaptive and robust dropping criterion: At loop $l\leq l_{\max}$,
a nonzero entry $m_{jk}$ is dropped if
\begin{equation}\label{dropping}
  |m_{jk}|\leq\frac{\eta}{nnz(m_{k})\|A\|_{1}},j=1,2,\ldots,n,
\end{equation}
where $nnz(m_{k})$ is the number of nonzero entries in $m_{k}$ and
$\|\cdot\|_1$ is the 1-norm of a matrix. From the theory in \cite{Jia13a},
PSAI($tol$) with the above dropping criterion will compute
a preconditioner $M$ that is as sparse as possible and meanwhile
has comparable preconditioning quality to the one generated by
BPSAI without dropping any nonzero entries. However, PSAI($tol$) involves large
sized LS problems for column irregular sparse matrices and is thus costly or
simply faces out of memory,
though $M$ obtained by it is an effective preconditioner \cite{Jia13b}.
As a whole, if $A$ is double irregular sparse, PSAI($tol$) is costly
and may be impractical.

\subsection{The RSAI($tol$) procedure}\label{sec:2.3}

We first review the basic RSAI (BRSAI) procedure \cite{Jia15}.
Suppose that $\mathcal{J}_{k}^{(l)}$ is the sparsity pattern of $m_{k}$
generated by BRSAI after $l$ loops starting with the initial sparsity
pattern $\mathcal{J}^{(0)}_{k}$. If $m_{k}$ does not yet satisfy the
prescribed accuracy $\eta$, BRSAI improves $m_{k}$ by
augmenting its sparsity pattern as follows.

Denote by $\mathcal{L}_{k}$ the set of indices $i$ for which the
$i$-th entry $r_{k}(i)\neq0$ of $r_k$.
BRSAI takes precedence to reduce a few largest entries in $r_{k}$ since they
make the most contributions to the size of $\|r_k\|$. The indices
corresponding to the largest entries $i$ of $r_k$ are called the dominant indices.
Denote by $\hat{\mathcal{R}}_{k}^{(l)}$
the set of the dominant indices $i$ with the largest entries $|r_{k}(i)|$,
and define $\hat{\mathcal{J}}_{k}$ to be the set of all new column indices of
$A$ that correspond to $\hat{\mathcal{R}}_{k}^{(l)}$
but do not appear in $\mathcal{J}_{k}^{(l)}$. Then we set
\begin{equation*}
  \mathcal{J}_{k}^{(l+1)}=\mathcal{J}_{k}^{(l)}\bigcup\hat{\mathcal{J}}_{k}.
\end{equation*}
When choosing $\hat{\mathcal{R}}_{k}^{(l)}$ from $\mathcal{L}_{k}$
in the above way, it may happen that
$\hat{\mathcal{R}}_{k}^{(l)}=\hat{\mathcal{R}}_{k}^{(l-1)}$.
If so, set
\begin{equation*}
  \mathcal{R}_{k}^{(l)}=\bigcup_{i=0}^{l-1}
  \hat{\mathcal{R}}_{k}^{(i)},
\end{equation*}
and choose $\hat{\mathcal{R}}_{k}^{(l)}$ from the set whose elements are in
$\mathcal{L}_{k}$ but not in $\mathcal{R}_{k}^{(l)}$. In this way,
$\hat{\mathcal{R}}_{k}^{(l)}$ is always non-empty unless $m_k$ is exactly
the $k$th column of $A^{-1}$ \cite{Jia15}. Denote by $\hat{\mathcal{I}}_{k}$
the set of indices of new nonzero rows corresponding to the added column
indices $\hat{\mathcal{J}}_{k}$. Then we update
$\mathcal{I}_{k}^{(l+1)}=\mathcal{I}_{k}^{(l)}\bigcup\hat{\mathcal{I}}_{k}$
and solve the new LS problem \eqref{updateLS}, which
generates a better approximation $m_k$ to the $k$th column of $A^{-1}$.
Repeat this process for $k=1,2,\ldots,n$ until $\|r_k\|\leq\eta$
or $l$ exceeds $l_{\max}$.

Similar to the PSAI($tol$) algorithm, in order to control the sparsity
of $M$, a practical RSAI($tol$) algorithm has been developed
in \cite{Jia15} that introduces the dropping criterion \eqref{dropping}
into BRSAI. It has been shown in \cite{Jia15} that RSAI($tol$) is as
equally effective as SPAI, but it is more efficient than SPAI
because we uses only a few other than all indices of the nonzero $r_k(i)$
and do not compute possibly a great many numbers $\rho_j$ in SPAI,
avoiding sorting them and picking up the most profitable indices. However,
RSAI($tol$) may be costly or out of memory for $A$ is row irregular sparse:
Suppose that the $k$th row of $A$ is relatively dense. Then once
$\{k\}\subseteq\hat{\mathcal{R}}_{k}^{(l)}$, the resulting $m_k$ is also
relatively dense, leading to a relatively large sized \eqref{reducemin2norm}.
Therefore, if $A$ is double irregular sparse,
RSAI($tol$) may be very costly and impractical.

\section{Transformation of double irregular sparse linear systems into double regular
sparse ones}\label{sec:3}

The previous discussion has indicated that SPAI, PSAI($tol$) and RSAI($tol$) is
costly and even impractical when $A$ is double irregular sparse.
In order to improve their efficiency, we will propose a 
transformation approach, which consists
four steps: (i)  transform a double irregular sparse
\eqref{equation} into a small number of double regular sparse ones with the same
coefficient matrix $\hat{A}$, (ii) use SPAI, PSAI($tol$)
and RSAI($tol$) efficiently construct possibly effective
sparse approximate inverses $M$ of $\hat{A}$, (iii)
solve the preconditioned double regular sparse linear systems
by Krylov solvers, and (iv) recover an approximate solution of
\eqref{equation} with a prescribed accuracy from those of the double
regular sparse ones. One of the key ingredients of the transformation approach
is the following well-known Sherman-Morrison-Woodbury
formula \cite[p. 50]{matrixcomputations}).

\begin{lemma}\label{theorem1}
Let $U,V\in R^{n\times s}$ with $s\leq n$. If $A$ is nonsingular, then
$A+UV^{T}$ is nonsingular if and only if $I+V^{T}A^{-1}U$ is nonsingular.
Furthermore,
\begin{equation}\label{shermanMOrrisonWoodbury}
  (A+UV^{T})^{-1} = A^{-1} -A^{-1}U(I+V^{T}A^{-1}U)^{-1}V^{T}A^{-1}.
\end{equation}
\end{lemma}

In practical applications, one is typically interested in the formula
for $s\ll n$, which reduces to the Sherman-Morrison formula when $s=1$.

In what follows, for the given double irregular sparse $A$ we
assume that the $j_1,j_2,\ldots,j_{s_1}$th columns and the
$i_1,i_2,\ldots,i_{s_2}$th rows of $A$ are
relatively dense, respectively. Let $A_{dc}=(a_{j_1},a_{j_2},\ldots,a_{j_{s_1}})$,
where $a_{j_k}$ is the $j_k$th column of $A$ and
$\tilde{A}_{dc}=(\tilde{a}_{j_1},\tilde{a}_{j_2},\ldots,\tilde{a}_{j_{s_1}})$
is the sparsification of $A_{dc}$, each column $\tilde{a}_{j_k}$
of which is sparse and retains $p$ nonzero entries of $a_{j_k}$, where
$p=\lfloor \frac{nnz(A)}{n}\rfloor$ is the average number
of nonzero entries per column of $A$ with $nnz(A)$ the number of nonzero
entries of $A$.
Let $U_1=A_{dc}-\tilde{A}_{dc}=(u_1,u_2,\ldots,u_{s_1})$. Obviously,
the nonzero entries of $U_1$ are just those dropped ones of $A_{dc}$.
Define $\tilde{A}$ to be the matrix that replaces the
dense columns $a_{j_k}$ of $A$ by the sparse vectors
$\tilde{a}_{j_k}$, $k=1,2,\ldots,s_1$.
Then $\tilde{A}$ is column regular sparse and satisfies
\begin{equation}\label{changecol}
  A = \tilde{A} + U_{1}V_{1}^{T},
\end{equation}
where $V_{1}=(e_{j_{1}},e_{j_{2}},\ldots,e_{j_{s_1}})$ with $e_{j_{k}}$
the $j_{k}$th column of $I$. Obviously, $U_{1}V_{1}^{T}$ is of rank $s_1$.
Assume that $\tilde{A}$ and $I+V_{1}^{T}\tilde{A}^{-1}U_{1}$ are
nonsingular. By \eqref{shermanMOrrisonWoodbury}, we
have
\begin{equation}\label{invcol}
  A^{-1} = \tilde{A}^{-1} - \tilde{A}^{-1}U_{1}(I+V_{1}^{T}\tilde{A}^
  {-1}U_{1})^{-1}V_{1}^{T}\tilde{A}^{-1}.
\end{equation}

Clearly, the $i_1,i_2,\ldots,i_{s_2}$th rows of $\tilde{A}$  are still dense.
Next, we further transform $\tilde{A}$ into double regular sparse. Let
$\tilde{A}_{dr}=(\tilde{a}_{i_1},\tilde{a}_{i_2},\ldots,\tilde{a}_{i_{s_2}})^{T}$,
where $\tilde{a}_{i_k}^T$ is the $i_k$th dense row of $\tilde{A}$, $k=1,2,\ldots,s_2$.
Define $\hat{A}_{dr}=(\hat{a}_{i_1},\hat{a}_{i_2},\ldots,\hat{a}_{i_{s_2}})^{T}$ to be
the sparsification of $\tilde{A}_{dr}$, where each row $\hat{a}_{i_k}^{T}$ is
sparse and retains only $\tilde{p}$ nonzero entries of $\tilde{a}_{i_k}^{T}$,
where $\tilde{p}=\lfloor \frac{nnz(\tilde{A})}{n}\rfloor$ is the average number
of nonzero entries per row of $\tilde{A}$.
Let $V_{2}=\tilde{A}_{dr}^{T}-\hat{A}_{dr}^{T}=(v_1,v_2,\ldots,v_{s_2})$.
Then the nonzero entries of $V_2$ are just those dropped ones of $\tilde{A}_{dr}$.
Let $\hat{A}$ be the matrix that replaces the dense rows $\tilde{a}_{i_k}^T$ of
$\tilde{A}$  by the sparse vectors $\hat{a}_{i_k}^T$, $k=1,2,\ldots,s_2$.
Then $\hat{A}$ is double regular sparse and is related to $\tilde{A}$ by
\begin{equation}\label{changerow}
  \tilde{A} = \hat{A} + U_2V_2^T,
\end{equation}
where $U_2=(e_{i_1},e_{i_2},\ldots,e_{i_{s_2}})$ with $e_{i_k}$ the $i_k$th
column of $I$ and $U_{2}V_{2}^{T}$ is of rank $s_2$.

The combination of \eqref{changecol} and \eqref{changerow}
gives
\begin{equation}\label{transcr}
\hat{A}=A-U_1V_1^T-U_2V_2^T.
\end{equation}
As a result, we have transformed the double irregular sparse $A$ into
the double regular sparse $\hat{A}$, which
modifies $A$ with two low rank $s_1$ and $s_2$ matrices $U_1V_1^T$
and $U_2V_2^T$, respectively.

Assume that $\hat{A}$ and $I+V_{2}^{T}\hat{A}^{-1}U_{2}$ are
nonsingular. By \eqref{shermanMOrrisonWoodbury}, we have
\begin{equation}\label{invrow}
  \tilde{A}^{-1} = \hat{A}^{-1} - \hat{A}^{-1}U_{2}(I+V_{2}^{T}\hat{A}^
  {-1}U_{2})^{-1}V_{2}^{T}\hat{A}^{-1}.
\end{equation}
From \eqref{invcol}, the solution of \eqref{equation} is
\begin{equation}\label{solution}
  x=A^{-1}b= \tilde{A}^{-1}b - \tilde{A}^{-1}U_{1}(I+V_{1}^{T}\tilde{A}^
  {-1}U_{1})^{-1}V_{1}^{T}\tilde{A}^{-1}b,
\end{equation}
which requires to compute
$y=\tilde{A}^{-1}b$ and $W=\tilde{A}^{-1}U_{1}$. Substituting \eqref{invrow}
into \eqref{solution}, we obtain
\begin{equation}\label{solutiontildeb}
  y=\tilde{A}^{-1}b=\hat{A}^{-1}b - \hat{A}^{-1}U_{2}(I+V_{2}^{T}\hat{A}^
  {-1}U_{2})^{-1}V_{2}^{T}\hat{A}^{-1}b
\end{equation}
and
\begin{equation}\label{solutiontildeW}
  W=\tilde{A}^{-1}U_{1}=\hat{A}^{-1}U_{1} - \hat{A}^{-1}U_{2}(I+V_{2}^{T}\hat{A}^
  {-1}U_{2})^{-1}V_{2}^{T}\hat{A}^{-1}U_{1},
\end{equation}
which require to compute $\hat{A}^{-1}b$, $\hat{A}^{-1}U_1$ and $\hat{A}^{-1}U_2$.
It is now clear that
from \eqref{invrow}--\eqref{solutiontildeW} that the ultimate computation
of $x$ needs $\hat{A}^{-1}b$, $\hat{A}^{-1}U_1$ and $\hat{A}^{-1}U_2$ as well
as the inversions of the small $s_1\times s_1$ matrix $ I+V_{1}^{T}\tilde{A}^{-1}U_{1}$
and $s_2\times s_2$ matrix $ I+V_{2}^{T}\hat{A}^{-1}U_{2}$, whose costs are
negligible relative to the computation of
$\hat{A}^{-1}b$, $\hat{A}^{-1}U_1$ and $\hat{A}^{-1}U_2$.

Keep in mind $U_1=(u_1,u_2,\ldots,u_{s_1})$ and $U_2=(e_{i_1},e_{i_2},
\ldots,e_{i_{s_2}})$. Then the computation of
$\hat{A}^{-1}b$, $\hat{A}^{-1}U_1$ and $\hat{A}^{-1}U_2$
amounts to solving the following $s_1+s_2+1$ double
regular sparse linear systems:
\begin{eqnarray}
 \hat{A}z&=&b, \label{equation-b}\\
\hat{A}p&=&u_{k},k=1,2,\ldots,s_1, \label{equation-col}\\
\hat{A}q&=&e_{i_{k}}, k =1,2,\ldots,s_2. \label{equation-row}
\end{eqnarray}

\begin{algorithm}
\caption{Solving the double irregular sparse linear system \eqref{equation}}
\begin{algorithmic}[1]\label{Procedure1}
\STATE Find $s_1$ and $A_{dc}$, and sparsify $A_{dc}$ to
       obtain $\tilde{A}_{dc}$.
       Let $U_1=A_{dc}-\tilde{A}_{dc}=(u_1,u_2,\ldots,u_{s_1})$ and
              $V_1=(e_{j_1},e_{j_2},\ldots,e_{j_{s_1}})$, and
       define $\tilde{A}=A-U_1V_1^T$.

\STATE Find $s_2$ and $\tilde{A}_{dr}$, and sparsify
       $\tilde{A}_{dr}$ to obtain
       $\hat{A}_{dr}$. Let $V_2=\tilde{A}_{dr}^T-\hat{A}_{dr}^T=(v_1,v_2,\ldots,
       v_{s_2})$ and $U_2=(e_{i_1},e_{i_2},\ldots,e_{i_{s_2}})$,
       and define $\hat{A}=\tilde{A}-U_2V_2^T$.

\STATE Solve the $s_1+s_2+1$ linear systems \eqref{equation-b},
       \eqref{equation-col} and \eqref{equation-row}
       for $z,p_1,p_2,\ldots,p_{s_1}$ and $q_1,q_2,\ldots,q_{s_2}$,
       respectively. Let $\hat{A}^{-1}U_1=P=(p_1,p_2,\ldots,p_{s_1})$
       and $\hat{A}^{-1}U_2=Q=(q_1,q_2,\ldots,q_{s_2})$. Compute $y$ by
       \begin{equation}\label{alg-y}
         y=\tilde{A}^{-1}b = z- Q(I+V_2^TQ)^{-1}(V_2^Tz)
       \end{equation}
       and $W$ by
       \begin{equation}\label{alg-W}
         W = P- Q(I+V_2^TQ)^{-1}(V_2^TP).
       \end{equation}
\STATE Compute the solution $x$ of $Ax=b$ by
       \begin{equation}\label{alg-Ax}
         x = A^{-1}b = y- W(I+V_1^TW)^{-1}(V_1^Ty).
       \end{equation}
\end{algorithmic}
\end{algorithm}

With the above derivation, we can now summarize our transformation approach to
solving $Ax=b$ as Algorithm 1. In our context, its implementation consists
of four steps:
Firstly, we construct the double regular sparse $\hat{A}$; secondly,
we compute a sparse approximate inverse $M$ of
$\hat{A}$ by SPAI, PSAI($tol$) or RSAI($tol$) and use it as a preconditioner for
the $s_1+s_2+1$ double regular sparse linear systems; thirdly, we solve
the preconditioned systems approximately by Krylov solvers; finally, we recover the
solution of \eqref{equation} from the $s_1+s_2+1$ solutions of the
double regular sparse linear systems. The three procedures SPAI, PSAI($tol$) and
RSAI($tol$) are expected to be much more efficient than them
applied to $A$ directly. On the other hand, as shown in \cite{Jia13b,Jia09},
as preconditioners, the $M$ constructed by the three procedures
applied to $\hat{A}$ are at least as effective as the
corresponding ones applied to $A$.

Finally, it is particularly worthwhile
to stress that, as is known from \cite{benzi02,benzi99,Ferronato12},
that the CPU time of constructing $M$ overwhelms that of Kryolv solvers for the
preconditioned systems even in a parallel computing environment,
provided that $M$ is an effective preconditioner. This is a typical feature of
F-norm minimization based and factorized sparse approximate inverse
preconditioning procedures.
Nonetheless, as is pointed out in \cite{benzi02,benzi99,Ferronato12},
although $F$-norm minimization sparse approximate inverse preconditioning
procedures are more costly than incomplete LU (ILU)
type preconditioning procedures, they are more robust, stable and general
than the latter ones are. Also importantly, the action of $M$
is to only form matrix-vector products, which is substantially advantageous to
ILU's where at each iteration of Krylov solvers one must solve
an auxiliary linear system with the given preconditioner as the coefficient matrix.
Lastly, $F$-norm minimization based sparse approximate inverse preconditioning
procedures are naturally parallelizable and suit better for the linear systems with
the same coefficient matrix but multiple right-hand sides, whereas ILU
preconditioning is inherently sequential and has very limited parallelization.
In view of these, the overall performance of Algorithm 1 is expected to
very greatly outperform the standard approach
that first preconditions \eqref{equation}
by SPAI, PSAI($tol$) or RSAI($tol$) directly and then solves it by Krylov
solvers.

\section{Theoretical analysis and practical considerations}\label{sec:4}

In this section, we will give some theoretical analysis and
practical considerations on the proposed transformation approach. Within the
framework of Algorithm 1, we then
develop a practical iterative solver for \eqref{equation}.
To this end, we need to consider several issues.

First of all, we adapt the definition of column irregular sparse
in \cite{Jia13b} to double irregular sparse: A column or row is claimed to be
dense if
the number of nonzero entries in it exceeds $10p$ or $10\tilde{p}$,
where $p=\lfloor \frac{nnz(A)}{n}\rfloor$ and $\tilde{p}
=\lfloor \frac{nnz(\tilde{A})}{n}\rfloor$ are the
average numbers of nonzero entries per column of $A$ and that per row
of $\tilde{A}$, respectively. The second issue is which nonzero
entries in $A_{dc}$ and $\tilde{A}_{dr}$
should be dropped to generate $\tilde{A}$ and $\hat{A}$.
We deal with this issue as follows:
Suppose that all the diagonals of $A$ are nonzero. We then adopt the
same strategy as that in \cite{Jia13b} and retain the
diagonal and the $p-1$ or $\tilde{p}-1$ nonzero entries nearest to the diagonal
in each column of $A_{dc}$ or each row of $\tilde{A}_{dr}$.

As we have seen, when transforming the double irregular sparse \eqref{equation}
into the double regular sparse ones, we always assume that the column regular
sparse $\tilde{A}$ and the double regular sparse $\hat{A}$ are nonsingular.
Our third issue arise naturally: for which classes of matrices $A$, one can
ensure the nonsingularity of $\hat{A}$? Jia and Zhang \cite{Jia13b} have
investigated
this problem in some detail for the column regular sparse $\tilde{A}$
and established a number of results, which are directly adapted to our current
double regular sparse $\hat{A}$, as Theorem~\ref{the1} and Corollary~\ref{the2}
states.

\begin{theorem}\label{the1}
$\hat{A}$ obtained by the above transformation approach is nonsingular for
the following classes of matrices:\\
(\rmnum{1}) $A$ is strictly row or column diagonally dominant.\\
(\rmnum{2}) $A$ is irreducibly row or column diagonally dominant.\\
(\rmnum{3}) $A$ is an $M$-matrix.\\
(\rmnum{4}) $A$ is an $H$--matrix.
\end{theorem}

The above four classes of matrices have wide applications. Similar to
\cite{Jia13b}, we can extend the first two classes of matrices
in Theorem~\ref{the1} to more general forms.

\begin{corollary}\label{the2}
$\hat{A}$ obtained by the above transformation approach is nonsingular for
the following classes of matrices:\\
(\rmnum{1}) $AD$ or $DA$ is strictly row or column diagonally dominant
            where $D$ is an arbitrary nonsingular diagonal matrix.\\
(\rmnum{2}) $PAQ$ is strictly row or column diagonally dominant
            where $P$ and $Q$ are permutation matrices. \\
(\rmnum{3}) $(PAQ)D$ or $D(PAQ)$ is strictly row or column diagonally dominant
            where $D$ is an arbitrary nonsingular diagonal matrix and $P,Q$ are
            permutation matrices. \\
(\rmnum{4}) $A$ is an irreducible analogue of the matrices in
            (\rmnum{1})--(\rmnum{3}).
\end{corollary}

There should exist more classes of matrices for which the resulting
$\hat{A}$ are nonsingular. We do not pursue this topic further. Strikingly,
we will numerically find that
for a general nonsingular $A$ that does not belong to
the aforementioned classes of matrices, the resulting $\hat{A}$
is indeed nonsingular.
This fact has been extensively verified for column irregular sparse
matrices and column regular sparse $\tilde{A}$ \cite{Jia13b}.

The fourth issue is the conditioning of $\hat{A}$.
Theoretically, for a general nonsingular $A$, the conditioning
of $\hat{A}$ may become better or worse. However, when $A$ is a strictly row or
column diagonally dominant matrix or a $M$-matrix, Jia and Zhang \cite{Jia13b} have given
mathematical justifications on why the resulting column regular sparse $\tilde{A}$
is generally better conditioned than $A$. The same arguments works for these classes of
matrices in our current context, and $\hat{A}$ can be shown to be generally better
conditioned than $A$. Remarkably, later numerical experiments will indicate that the
resulting $\hat{A}$ is always and often considerably better conditioned
than a general double
irregular sparse $A$ which does not fall into the matrices in Theorem~\ref{the1}
and Corollary~\ref{the2}.

The above considerations on the third and fourth issues indicates
that our transformation approach is of generality in practical
applications.

Our final issue is the selection of stopping criteria for Krylov
solvers applied to the $s_1+s_2+1$ double regular sparse
linear systems for a given the prescribed accuracy $\varepsilon$
for \eqref{equation}. This selection is crucial for reliably recovering an
approximate solution of \eqref{equation} with the prescribed accuracy $\varepsilon$.
In order to solve this problem, we present the following theorem, based on which
we can design reliable stopping criteria for the double regular sparse linear systems.

\begin{theorem}\label{theorem-stoppingcriteria}
For $U_1, U_2, V_1$ and $V_2$ defined in Algorithm 1,
let $\check{z}$, $\check{p}_{j}$, $j=1,2,\ldots,s_1$ and $\check{q}_{j}$,
$j=1,2,\ldots,s_2$, be the approximate solutions of \eqref{equation-b},
\eqref{equation-col} and \eqref{equation-row}, respectively.
Define $\check{P}=(\check{p}_{1},\check{p}_{2},\ldots,\check{p}_{s_1})$,
$\check{Q}=(\check{q}_{1},\check{q}_{2},\ldots,\check{q}_{s_2})$ and
the residuals $r_{\check{z}}=b-\hat{A}\check{z}$,
$r_{\check{p}_{j}}=u_{j}-\hat{A}\check{p}_{j}$,
$r_{\check{q}_{j}}=e_{i_{j}}-\hat{A}\check{q}_{j}$,
assume that $I+V_{2}^{T}\check{Q}$ is nonsingular, and define
$c_{1}=\|(I+V_{2}^{T}\check{Q})^{-1}V_{2}^{T}\check{z}\|$ and
$c_{2}=\|(I+V_{2}^{T}\check{Q})^{-1}V_{2}^{T}\check{P}\|$, and
take
\begin{equation}\label{app-y}
\check{y}=\check{z}-\check{Q}(I+V_{2}^{T}\check{Q})^{-1}V_{2}^{T}\check{z}
\end{equation}
and
\begin{equation}\label{app-W}
\check{W}=\check{P}-\check{Q}(I+V_{2}^{T}\check{Q})^{-1}V_{2}^{T}\check{P}
\end{equation}
to be the approximations of $y$ and $W$ defined by \eqref{solutiontildeb}
and \eqref{solutiontildeW}, respectively.
Assume that $I+V_{1}^{T}\check{W}$ is nonsingular, and
define $c_{0}=\|(I+V_{1}^{T}\check{W})^{-1}V_{1}^{T}\check{y}\|$ and
\begin{equation}\label{app-x}
\check{x}=\check{y}-\check{W}(I+V_{1}^{T}\check{W})^{-1}V_{1}^{T}\check{y}
\end{equation}
to be an approximate solution of $Ax=b$. Then if
\begin{equation}\label{epsy}
  \frac{\|r_{\check{z}}\|}{\|b\|}\leq\frac{\varepsilon}{4},
\end{equation}
\begin{equation}\label{epscol}
  \frac{\|r_{\check{p}_{j}}\|}{\|u_{j}\|}\leq\frac{\|b\|}{4\sqrt{s_1}c_{0}
  \|u_{j}\|}\varepsilon, j=1,2,\ldots,s_1
\end{equation}
and
\begin{equation}\label{epsrow}
  \frac{\|r_{\check{q}_{j}}\|}{\|e_{i_{j}}\|}\leq
  \frac{\|b\|}{2\sqrt{s_2}(c_{0}c_{2}+c_{1})}
  \varepsilon, j=1,2,\ldots,s_2,
\end{equation}
we have
\begin{equation}\label{epsx}
  \frac{\|r\|}{\|b\|}= \frac{\|b-A\check{x}\|}{\|b\|}
  \leq\varepsilon.
\end{equation}
\end{theorem}
\begin{proof}
Define $R_{\check{W}}=U_{1}-\tilde{A}\check{W}$ and
$r_{\check{y}}=b-\tilde{A}\check{y}$.
From \eqref{changecol}, we obtain
\begin{align}\label{relationr}
  r&=b-A\check{x}=b-(\tilde{A}+U_{1}V_{1}^{T})(\check{y}-
    \check{W}(I+V_{1}^{T}\check{W})^{-1}V_{1}^{T}\check{y}) \notag \\
   &=b-\tilde{A}\check{y}-U_{1}V_{1}^{T}\check{y}+\tilde{A}
   \check{W}(I+V_{1}^{T}\check{W})^{-1}V_{1}^{T}\check{y}+U_{1}
   V_{1}^{T}\check{W}(I+V_{1}^{T}\check{W})^{-1}V_{1}^{T}\check{y} \notag \\
   &=r_{\check{y}}-U_{1}[(I+V_{1}^{T}\check{W})-V_{1}^{T}\check{W}]
   (I+V_{1}^{T}\check{W})^{-1}V_{1}^{T}\check{y}+\tilde{A}\check{W}
   (I+V_{1}^{T}\check{W})^{-1}V_{1}^{T}\check{y}  \notag \\
   &=r_{\check{y}}-U_{1}(I+V_{1}^{T}\check{W})^{-1}V_{1}^{T}\check{y}
   +\tilde{A}\check{W}(I+V_{1}^{T}\check{W})^{-1}V_{1}^{T}\check{y} \notag \\
   &=r_{\check{y}}-R_{\check{W}}(I+V_{1}^{T}\check{W})^{-1}V_{1}^{T}\check{y}.
\end{align}


Define $R_{\check{P}}=U_{1}-\hat{A}\check{P}$ and
$R_{\check{Q}}=U_{2}-\hat{A}\check{Q}$. From \eqref{app-y} and \eqref{app-W}
we obtain
\begin{equation}\label{relationry}
  r_{\check{y}}=r_{\check{z}}-R_{\check{Q}}(I+V_{2}^{T}
  \check{Q})^{-1}V_{2}^{T}\check{z}
\end{equation}
and
\begin{equation}\label{relationrw}
  R_{\check{W}}=R_{\check{P}}-R_{\check{Q}}(I+V_{2}^{T}
  \check{Q})^{-1}V_{2}^{T}\check{P},
\end{equation}
respectively. From \eqref{relationr}, \eqref{relationry} and
\eqref{relationrw}, by the definitions of $c_{0}$, $c_{1}$ and
$c_{2}$ we have
\begin{align}
  \|r\| & \leq\|r_{\check{y}}\|+c_{0}\|R_{\check{W}}\|
        \leq\|r_{\check{z}}\|+c_{1}\|R_{\check{Q}}\|+c_{0}
             (\|R_{\check{P}}\|+c_{2}\|R_{\check{Q}}\|) \notag\\
        &\leq\|r_{\check{z}}\|+c_{0}\|R_{\check{P}}\|_{F}+
             (c_{0}c_{2}+c_{1})\|R_{\check{Q}}\|_{F}. \label{norm-r}
\end{align}
Noting that $R_{\check{P}}(:,j)=r_{\check{p}_{j}}$ and
$R_{\check{Q}}(:,j)=r_{\check{q}_{j}}$, we obtain
$\|R_{\check{P}}\|_{F}=\sqrt{\sum_{j=1}^{s_1}\|r_{\check{p}_{j}}\|^{2}}$
and $\|R_{\check{Q}}\|_{F}=\sqrt{\sum_{j=1}^{s_2}\|r_{\check{q}_{j}}\|^{2}}$.
It is then easy to verify that if \eqref{epsy}, \eqref{epscol} and \eqref{epsrow}
are fulfilled then \eqref{epsx} follows from \eqref{norm-r}.
\end{proof}

Though Theorem~\ref{theorem-stoppingcriteria} gives
the stopping criteria for the double regular sparse problems,
they are not directly applicable for the prescribed accuracy
$\varepsilon$ of \eqref{equation}. The reason is that
$c_{0}$, $c_{1}$ and $c_{2}$ can be computed only until iterations for
\eqref{equation-b}, \eqref{equation-col} and \eqref{equation-row}
terminate, but the stopping criteria \eqref{epscol} and \eqref{epsrow}
for \eqref{equation-col} and \eqref{equation-row} depend on  $c_{0}$, $c_{1}$
and $c_{2}$. As a result, the computation of $c_{0}$, $c_{1}$ and $c_{2}$ and
the termination of iterative solvers interact, and cannot be done
in advance. Fortunately, it appears that the accurate estimates of
$c_{0}$, $c_{1}$ and $c_{2}$ are unnecessary and quite rough ones are enough.
It is seen that $c_{0}$ is moderate if $I+V_{1}^{T}\check{W}$ is not ill
conditioned. Suppose that $I+V_{1}^{T}\check{W}$ and $I+V_{2}^{T}\check{Q}$
are not very ill conditioned. We observe that $c_{1}$ and $c_{2}$ rely on the
norm of $V_2$ and $c_0$ depends on the norm of $V_1$.
In implementations, we simply take $c_{0}=1$ and $c_{1}=c_{2}=
\max_{1\leq i\leq s_2}\|V_2(:,i)\|$.
In numerical experiments, we will find that such choices of $c_{0}$, $c_{1}$ and
$c_{2}$ work reliably and makes \eqref{epsx} hold for all the test problems.

Having done the above, we have finally developed Algorithm 1 into a truly working
algorithm, called Algorithm 2 hereafter.

\section{Numerical experiments}\label{sec:5}

In this section, we test Algorithm 2 on ten real-world double irregular
sparse nonsymmetric problems listed in Table~\ref{table-mtr}, which are
from \cite{davis11}. We first construct sparse
approximate inverses $M$ of double regular sparse matrices $\hat{A}$
by SPAI, PSAI($tol$) or RSAI($tol$) and then solve the resulting
$s_1+s_2+1$ preconditioned double regular sparse linear systems by the Krylov
solver BiCGStab, whose code {\sf bicgstab.m} is from Matlab 7.8.0.
We will compare Algorithm 2 with the standard approach that preconditions
\eqref{equation} by SPAI, PSAI($tol$) or RSAI($tol$) directly and solves
the preconditioned double irregular sparse system by BiCGStab.
Depending on which of SPAI, PSAI($tol$) and RSAI($tol$) is used, Algorithm 2
gives rise to three algorithms, named New-SPAI, New-PSAI($tol$) and New-RSAI($tol$),
abbreviated as N-SPAI, N-PSAI($tol$) and N-RSAI($tol$), respectively.
Similarly, we denote by S-SPAI, S-PSAI($tol$), and S-RSAI($tol$) the algorithms
that directly precondition \eqref{equation} by SPAI, PSAI($tol$) and RSAI($tol$),
respectively.

We will make numerous comparisons. Most importantly,
in terms of the CPU time, we shall
demonstrate that N-SPAI, N-PSAI($tol$) and N-RSAI($tol$)
outperform S-SPAI, S-PSAI($tol$) and S-PSAI($tol$) very greatly, respectively.
Particularly, we will show that S-SPAI, S-PSAI($tol$) and S-PSAI($tol$)
fail to compute preconditioners $M$ for the last six larger ones of the ten
test problems because of huge computational cost or memory storage,
but N-SPAI, N-PSAI($tol$) and N-RSAI($tol$) work very well.

\begin{table}[!htb]
\centering
\footnotesize
\caption{The description of test matrices}\label{table-mtr}
\begin{tabular}{cccc}\hline
matrices&$n$&$nnz(A)$&Description\\
\hline
rajat04&1,041&8,725&circuit simulation problem\\
rajat12&1,879&12,818&circuit simulation problem\\
rajat13&7,598&48,762&circuit simulation problem\\
memplus&17,758&99,147&computer component design memory circuit\\
ASIC\_100k&9,9340&940,621&Sandia, Xyce circuit simulation matrix\\
dc1&116,835&766,396&circuit simulation problem\\
dc2&116,835&766,396&circuit simulation problem\\
dc3&116,835&766,396&circuit simulation problem\\
trans4&116,835&749,800&circuit simulation problem\\
trans5&116,835&749,800&circuit simulation problem\\
\hline
\end{tabular}
\end{table}

We perform numerical experiments on an Intel Core 2
Quad CPU E8400@ 3.00GHz with 2GB RAM using Matlab 7.8.0
with the machine precision $\epsilon_{\rm mach}=2.22\times10^{-16}$ under
the Linux operating system. We use the SPAI 3.2 package \cite{spai}
for the SPAI algorithm, which is written in C/MPI and is an optimized code in
some sense.  PSAI($tol$) and RSAI($tol$) are experimental Matlab codes in
the sequential environment.
We take the initial sparsity pattern as that of $I$ for SPAI and RSAI($tol$).
We apply row Dulmage-Mendelsohn permutations to the matrices
having zero diagonals so as to make their diagonals nonzero \cite{duff}.
The related Matlab command is $j={\sf demperm}(A(j,:))$, which is applied
to rajat04, rajat12, rajat13 and ASIC\_100k.
We take $c_0$, $c_1$ and $c_2$ as the values defined in the end of
Section~\ref{sec:4}, and stop BiCGStab when \eqref{epsy}, \eqref{epscol} and
\eqref{epsrow} are satisfied with $\varepsilon=10^{-8}$ or 1,000 iterations
are used. The initial guess of the solution to each problem is always $x_{0}=0$
and the right-hand side $b$ is formed by choosing the solution
$x=(1,1,\ldots,1)^{T}$. We have found that BiCGStab without preconditioning
does not converge for all the test problems in Table~\ref{table-mtr}
within 1,000 iterations except for rajat04, rajat12 and rajat13.
With $\check{x}$ defined by \eqref{app-x}, we compute the actual relative
residual norm
\begin{equation}\label{relativerenorm}
  r_{actual}=\frac{\|b-A \check{x}\|}{\|b\|}
\end{equation}
and compare it with the required accuracy $\varepsilon=10^{-8}$.

In Table~\ref{table-mtrvs}, we give some information on $A$
and $\hat{A}$ constructed by Algorithm 2. It is observed that
all the test matrices have some almost fully dense columns and rows
except rajat04, rajat12, rajat13 and memplus.
It is also seen from the table that $s_1$ and $s_2$ are very small relative
to $n$, as they must be. Having sparsified those dense columns and rows,
we find that the number of nonzero entries in $\hat{A}$ are considerably
smaller than those in $A$. We remark that all the test matrices do not belong
to the classes of matrices in Theorem~\ref{the1}, but the table shows that
all the $\hat{A}$ are always better and can be
better conditioned than the corresponding $A$ by one to nearly three orders,
as the condition numbers
$\kappa(A)$ and $\kappa(\hat{A})$ indicate clearly. This
demonstrates that our transformation of $A$ into $\hat{A}$ is of practical
generality that ensures not only the non-singularity of $\hat{A}$ but also
improves the conditioning of the double regular sparse linear systems.

\begin{table}[!htb]
\centering
\footnotesize
\tabcolsep 4pt
\caption{\label{table-mtrvs}\small
{Some information on $A$ and $\hat{A}$, where
$s_1$ and $s_2$ denote the numbers of irregular columns and rows of $A$,
respectively, $p=\lfloor \frac{nnz(A)}{n}\rfloor$, $p_{dc}$ and $p_{dr}$
the numbers of nonzero entries in the densest column of $A$ and
the densest row of $\tilde{A}$,
$\nu=\max_{1\leq i\leq s_2}\|V_2(:,i)\|$, and the condition number
$\kappa(A)=\|A\|\|A^{-1}\|$.
The Matlab function {\sf condest.m} is used to estimate the 1-norm
condition numbers of the last six larger matrices.}}
\begin{tabular}{|c|c|c|c|c|c|c|c|c|c|}\hline
matrices&$s_1$&$s_2$&$p$&$p_{dc}$&$p_{dr}$
&$nnz(\hat{A})$&$\nu$&$\kappa(A)$&$\kappa(\hat{A})$\\
\hline
rajat04&5&6&8&642&659&5702
   &$6.33\times 10^{3}$&$1.64\times 10^{8}$&$6.85\times 10^{7}$\\
\hline
rajat12&9&9&6&1,195&1,190&6,803
   &$2.26\times 10^{2}$&$6.91\times 10^{5}$&$6.36\times 10^{5}$\\
\hline
rajat13&29&29&6&5,412&5,383&27,232
   &$3.08\times 10^{3}$&$1.19\times 10^{11}$&$3.59\times 10^{8}$\\
\hline
memplus&144&124&6&353&319&67,649
   &$1.45\times 10^{-1}$&$1.29\times 10^{5}$&$1.24\times 10^{5}$\\
\hline
ASIC\_100k&132&129&9&92,258&92,165&543,876
   &$3.93\times 10^{1}$&$1.46\times 10^{11}$&$9.27\times 10^{9}$\\
\hline
dc1&55&54&6&114,174&114,184&425,819
   &$5.59\times 10^{4}$&$1.01\times 10^{10}$&$2.21\times 10^{8}$\\
\hline
dc2&55&54&6&114,174&114,184&425,819
   &$5.74\times 10^{4}$&$8.86\times 10^{9}$&$5.89\times 10^{7}$\\
\hline
dc3&55&54&6&114,174&114,184&425,819
   &$6.00\times 10^{4}$&$1.16\times 10^{10}$&$1.53\times 10^{8}$\\
\hline
trans4&55&54&6&114,174&114,184&425,819
   &$1.91\times 10^{4}$&$3.30\times 10^{9}$&$3.46\times 10^{8}$\\
\hline
trans5&55&54&6&114,174&114,184&425,819
   &$6.01\times 10^{3}$&$2.32\times 10^{9}$&$6.54\times 10^{7}$\\
\hline
\end{tabular}
\end{table}

\begin{figure}[!htb]
\includegraphics{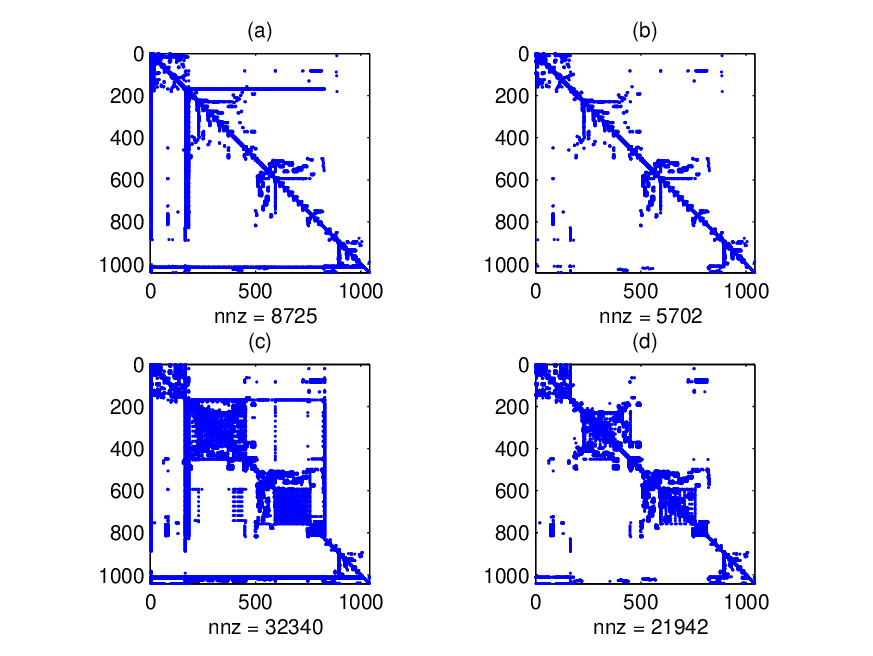}
\caption{(a): the sparsity pattern of $A$;
(b): the sparsity pattern of $\hat{A}$;
(c): the sparsity pattern of the sparsified $A^{-1}$;
(d): the sparsity pattern of the sparsified $\hat{A}^{-1}$.}\label{figure1}
\end{figure}

Now we look into the effective approximate sparsity pattern of $A^{-1}$ and
that of $\hat{A}^{-1}$ obtained by our transformation approach.
We aim to show that effective sparse approximate inverses of double
irregular and regular sparse matrices are structure preserving,
though theoretically good approximate inverses of
a double regular sparse matrix may be irregular sparse \cite{Jia13b}.
We take rajat04 as an example. Performing a row
Dulmage-Mendelsohn permutation on it, we depict the sparsity patterns of
$A$ and $\hat{A}$. We first use the Matlab function {\sf inv} to compute
$A^{-1}$ and $\hat{A}^{-1}$ accurately and then drop their nonzero entries
whose magnitudes fall below $10^{-3}$ so as to obtain good sparse approximate
inverses of $A$ and $\hat{A}$. Figure~\ref{figure1} depicts
the patterns of $A$, $\hat{A}$, and the sparsified $A^{-1}$ and
$\hat{A}^{-1}$. Clearly, the generated good approximate inverse of $A$
is sparse, but it is double irregular sparse, whose numbers of irregular columns
and rows are no less than those of $A$. In contrast,
for the double regular sparse $\hat{A}$, the generated
good sparse approximate inverse of
it is not only double regular sparse but also considerably sparser than that
of $A$. These observations imply that SPAI is not only costly but also
cannot construct effective preconditioners $M$ of \eqref{equation} since
the $M$ by SPAI are column regular sparse unless the loops $l_{\max}$ are allowed
to be very big, which is prohibited for SPAI. PSAI($tol$) and RSAI($tol$) are
also costly. In contrast, since good sparse approximate
inverses of $\hat{A}$ are generally double regular sparse, it is expected
that SPAI, PSAI($tol$) and RSAI($tol$) construct good approximate inverses $M$ of
$\hat{A}$ much more efficiently than they do for $A$.

In the later tables, we denote by $spar=\frac{nnz(M)}{nnz(A)}$ or
$\frac{nnz(M)}{nnz(\hat{A})}$
the sparsity of $M$ relative to $A$ or $\hat{A}$, by
$n_{c}$ the number of columns of $M$ whose
residual norms do not drop below the prescribed accuracy
$\eta$, by $ptime$ and $stime$ the CPU time (in seconds) of constructing
$M$ and that of solving the preconditioned
linear systems by BiCGStab, respectively. The notation $\dag$ means that we do
not count CPU time when BiCGStab fails to converge within 1000 iterations.
Let the actual relative residual
norm $r_{actual}=a\cdot\varepsilon$. Then $a<1$ means that our
choices of $c_0,c_1$ and $c_2$ work reliably and
relative residual norm
\eqref{relativerenorm} satisfies the prescribed accuracy $\varepsilon$.

\subsection{S-SPAI and N-SPAI}\label{sec:5.1}

We will show that N-SPAI is much more efficient
than S-SPAI with the same parameters used in SPAI, in which the
initial pattern of $M$ is that of $I$, and we take $\eta=0.4$, $l_{\max}=20$
and add five most profitable indices to the pattern of $m_k$ at each loop
for $k=1,2,\ldots,n$. The number of nonzero entries
in $m_k$ is therefore no more than $1+5\times 20 =101$ for
the given parameters. Consequently, if
good preconditioners have at least one column whose number of
nonzero entries is bigger than 101, SPAI may be ineffective
for preconditioning \eqref{equation}.
Table~\ref{table-spai} shows the results, where the notation $\ast$
indicates that S-SPAI could not construct $M$ when 100 hours are
consumed, and the notation $iter$ stands for the numbers of iterations
and maximum iterations that BiCGStab uses for \eqref{equation} and
the $s_1+s_2+1$ systems \eqref{equation-b}, \eqref{equation-col}
and \eqref{equation-row}, respectively.

\begin{table}[!htb]
\centering
\footnotesize
\tabcolsep 5pt
\caption{\label{table-spai}S-SPAI versus N-SPAI}
\begin{tabular}{|c|c|c|c|c|c|c|c|c|c|c|c|c|}\hline
&\multicolumn{6}{|c|}{S-SPAI}&\multicolumn{6}{|c|}
{N-SPAI}\\
\cline{2-7}\cline{8-13}
matrices&$spar$&$ptime$&$n_{c}$&$a$&$iter$&$stime$&$spar$&$ptime$&$n_{c}$
&$a$&$iter$&$stime$\\
\hline
rajat04&0.37&1.18&6&0.33&30&0.14&0.39&0.05&2&0.27&15&0.07\\
\hline
rajat12&0.89&3.19&3&0.58&46&0.04&0.81&0.03&0&0.35&38&0.25\\
\hline
rajat13&0.92&271.2&6&0.91&73&0.19&1.16&0.27&2&0.12&8&0.35\\
\hline
memplus&1.05&13.6&0&0.70&92&0.40&1.35&0.48&0&0.50&23&4.06\\
\hline
ASIC\_100k&$\ast$&$\ast$&$\ast$&$\ast$&$\ast$&$\ast$&0.55&4.36&12&0.35&9&27.0\\
\hline
dc1&$\ast$&$\ast$&$\ast$&$\ast$&$\ast$&$\ast$&1.10&8.80&6&0.31&314&95.8\\
\hline
dc2&$\ast$&$\ast$&$\ast$&$\ast$&$\ast$&$\ast$&1.04&7.70&0&0.40&107&51.6\\
\hline
dc3&$\ast$&$\ast$&$\ast$&$\ast$&$\ast$&$\ast$&1.06&8.01&6&0.28&81&55.8\\
\hline
trans4&$\ast$&$\ast$&$\ast$&$\ast$&$\ast$&$\ast$&1.14&8.05&0&0.26&37&10.0\\
\hline
trans5&$\ast$&$\ast$&$\ast$&$\ast$&$\ast$&$\ast$&1.13&9.00&0&0.47&79&20.1\\
\hline
\end{tabular}
\end{table}

From Table~\ref{table-spai}, we find that all the $a<1$ for all the test problems.
This indicates that our choices of $c_0,c_1$ and $c_2$ are
reliable in practice. For the last six larger problems S-SPAI could
not construct $M$ within 100 hours but N-SPAI does the job in no more than
nine seconds, and Algorithm 2 solves all the preconditioned double regular
sparse linear systems with the total CPU time ($=ptime+stime$) between
$18.1\sim 104.6$ seconds, very dramatic
improvements over the standard approach!
The reason is that each of these matrices contains
some fully dense columns and rows, so that SPAI spends unaffordable time
in finding most profitable indices because of the large cardinalities
of $\hat{\mathcal{J}}_{k}$ and $\mathcal{N}_{k}$ for
$k=1,2,\ldots,n$ at each loop. Precisely, for a double irregular sparse
matrix whose irregular column and row are fully dense,  at each loop
$l$, SPAI has to compute almost $n$ numbers $\rho_j$ by
\eqref{2normofresidual}, then sort almost $n$ indices in
$\hat{\mathcal{J}}_{k}$ by comparing the sizes of
$\rho_j$ and finally pick up a few most profitable indices
among them. This is a huge computational task and makes SPAI fatally slow.
In contrast, N-SPAI overcomes this drawback very well since
there are only a very small number of elements in $\hat{\mathcal{J}}_{k}$ and
N-SPAI only needs to compute the same number of $\rho_j$ and select the
most profitable indices. This is why N-SPAI outperforms
S-SPAI so dramatically. For each of the smaller rajat04, rajat12,
rajat13 and the relatively large memplus whose irregular rows and columns
are, though relatively dense, far from fully dense, we observe
that the $ptime$ by N-SPAI is still
much smaller than the corresponding one by S-SPAI and
the CPU time is reduced by tens to one thousand of times.
So, even for $A$ that does not have very dense irregular columns and rows,
N-SPAI exhibits its much higher efficiency than S-SPAI does.

We next look at the preconditioning effectiveness of $M$.
We observe that the $n_{c}$ in S-SPAI are bigger than those counterparts
in N-SPAI for rajat04, rajat12, rajat13 and memplus.
This demonstrates that SPAI is difficult
to capture a good approximate sparsity pattern of $A^{-1}$ when $A$
is double irregular sparse. This is also confirmed by the numbers
$iter$'s, which show that BiCGStab converges much
faster for the the double regular sparse problems than
for \eqref{equation}, that is, SPAI is considerably less effective for
preconditioning double irregular sparse linear systems than it is
for double regular sparse ones.

Finally, we compare our transformation approach with that in \cite{Jia13b}, where
double irregular sparse matrices are only transformed into column regular sparse
ones, which are still row irregular sparse for our test matrices.
Therefore, SPAI is still time consuming, though its efficiency is improved
greatly relative to its application to $A$ directly. We mention
that for numerical experiments we use the same computer as that in \cite{Jia13b}.
Precisely, for the last six larger matrices, the approach used in \cite{Jia13b}
takes about half an hour to two hours to construct the $M$, while our approach
here only costs about eight seconds, improvements of hundreds of times!

\subsection{S-PSAI($tol$) and N-PSAI($tol$)}\label{sec:5.3}

We will illustrate that N-PSAI($tol$) is much more efficient than
S-PSAI($tol$), where we take $\eta=0.4$ and $l_{\max}=10$.
Table~\ref{table-psai} shows the results, where the notation $-$ indicates that
our computer is out of memory when constructing $M$ due to the appearance
of large sized LS problems.

\begin{table}[!htb]
\tabcolsep 5pt
\centering
\footnotesize
\caption{\label{table-psai}S-PSAI($tol$) versus N-PSAI($tol$)}
\begin{tabular}{|c|c|c|c|c|c|c|c|c|c|c|c|c|}\hline
&\multicolumn{6}{|c|}{S-PSAI($tol$)}&\multicolumn{6}{|c|}
{N-PSAI($tol$)}\\
\cline{2-7}\cline{8-13}
matrices&$spar$&$ptime$&$n_{c}$&$a$&$iter$&$stime$&$spar$&$ptime$&$n_{c}$
&$a$&$iter$&$stime$\\
\hline
rajat04&0.72&1.40&0&0.79&11&0.01&0.39&0.18&0&0.15&13&0.08\\
\hline
rajat12&2.22&3.48&0&0.74&32&0.03&2.11&0.80&0&0.28&37&0.19\\
\hline
rajat13&1.36&204&0&0.19&4&0.07&0.91&0.17&0&0.29&6&0.27\\
\hline
memplus&6.25&624&0&0.98&215&1.30&1.78&48.6&0&0.35&27&4.08\\
\hline
ASIC\_100k&$-$&$-$&$-$&$-$&$-$&$-$&0.43&582&0&0.34&8&27.7\\
\hline
dc1&$-$&$-$&$-$&$-$&$-$&$-$&1.67&1340&0&0.26&379&101\\
\hline
dc2&$-$&$-$&$-$&$-$&$-$&$-$&1.71&1295&0&0.36&58&48.3\\
\hline
dc3&$-$&$-$&$-$&$-$&$-$&$-$&1.73&1299&0&0.47&54&44.0\\
\hline
trans4&$-$&$-$&$-$&$-$&$-$&$-$&1.80&1679&0&0.26&28&10.9\\
\hline
trans5&$-$&$-$&$-$&$-$&$-$&$-$&1.62&1475&0&0.73&60&17.8\\
\hline
\end{tabular}
\end{table}

From the table, we observe that all the $a<1$ for all the test problems.
This indicates that our choices of $c_0,c_1$ and $c_2$ are reliable.
Clearly, the table tells us that for the last six larger problems
S-PSAI($tol$) could not construct $M$. This is because
that each of these $A$ has fully dense columns, which lead to some large
LS problems \eqref{reducemin2norm} with dimensions $p_{dc}\approx n$
and generate some fully dense $m_k$ before dropping small nonzero entries
in them; see \cite{Jia13b,Jia09} for details.
So, for $n$ large, PSAI($tol$) faces a severe difficulty when applied to
column irregular sparse matrices. In contrast, with
the same parameters, N-PSAI($tol$) overcomes
this difficulty very well and computes effective preconditioners $M$ efficiently
for the last six larger problems. By comparison, for rajat04, rajat12, rajat13
and memplus, N-PSAI($tol$) can construct the $M$ several or many times
faster than S-PSAI($tol$).
However, unlike SPAI, we observe all $n_c=0$, indicating that both S-PSAI($tol$) and
N-PSAI($tol$) succeed in finding effective sparse approximate inverses,
which confirms the theory that
PSAI($tol$) can capture effective approximate sparsity patterns of the inverse of
a sparse matrix, independent of whether the matrix is regular or irregular sparse
\cite{Jia13b,Jia09}.

Next, we compare our transformation approach with that used
in \cite{Jia13b}, which only transforms double irregular sparse $A$ into column
regular sparse ones. We point out that good sparse approximate inverses of a row
irregular sparse matrix may have some relative dense columns;
see Figure 5.2 of \cite{Jia13b} for rajat04. In this case, PSAI($tol$) needs to
solve some LS problems whose sizes are not as small as those
for double regular sparse matrices, which causes
PSAI($tol$) to be considerably more costly for only column regular sparse
matrices than for double regular sparse ones. Indeed, in comparison with
the results obtained by the approach in \cite{Jia13b}, we find that for the
last six larger matrices our approach saves about half the CPU time $ptime$.

Finally, we make some comments on N-SPAI and N-PSAI($tol$). From
Tables~\ref{table-spai}--\ref{table-psai}, we find that the $spar$ by N-SPAI
and N-PSAI($tol$) are correspondingly comparable
but N-PSAI($tol$) is at least competitive with N-SPAI and
the former is considerably more effective
than the latter for preconditioning half of the test
problems, as indicated by the corresponding $n_c$ and $iter$.
This justifies
that PSAI($tol$) is more effective than SPAI to capture good sparsity patterns
of approximate inverses even for double regular sparse matrices. In
addition, we have noticed that N-PSAI($tol$) is more costly than N-SPAI.
This is simply due to our non-optimized code of PSAI($tol$) in the Matlab language,
whose efficiency is inferior to the optimized SPAI code written in C/MPI.

\subsection{S-RSAI($tol$) and N-RSAI($tol$)}\label{sec:5.2}

We will demonstrate that N-RSAI($tol$) is much more efficient than
S-RSAI($tol$). We take $\eta=0.4$ and $l_{\max}=10$, and
use {\em three} dominant indices $i$ with the largest $|r_k(i)|$ at
each loop for $k=1,2,\ldots,n$. Table~\ref{table-rsai} lists the results,
where the notation $*$ indicates that S-RSAI($tol$) could not construct
$M$ within 25 hours and $-$ indicates that
our computer is out of memory when constructing $M$ due to the appearance
of large sized LS problems resulting from the relatively dense rows of $A$.

\begin{table}[!htb]
\tabcolsep 5pt
\centering
\footnotesize
\caption{\label{table-rsai}S-RSAI($tol$) versus N-RSAI($tol$)}
\begin{tabular}{|c|c|c|c|c|c|c|c|c|c|c|c|c|}\hline
&\multicolumn{6}{|c|}{S-RSAI($tol$)}&\multicolumn{6}{|c|}
{N-RSAI($tol$)}\\
\cline{2-7}\cline{8-13}
matrices&$spar$&$ptime$&$n_{c}$&$a$&$iter$&$stime$&$spar$&$ptime$&$n_{c}$
&$a$&$iter$&$stime$\\
\hline
rajat04&7.26&9.05&3&0.34&17&0.05&0.52&0.31&3&0.31&19&0.09\\
\hline
rajat12&129&110&3&0.37&11&0.16&2.06&1.05&0&0.18&8&0.09\\
\hline
rajat13&$*$&$*$&$*$&$*$&$*$&$*$&3.85&13.6&2&0.12&8&0.39\\
\hline
memplus&8.81&675&0&0.63&16&0.18&1.73&39.0&0&0.53&16&3.09\\
\hline
ASIC\_100k&$-$&$-$&$-$&$-$&$-$&$-$&1.54&1292&0&0.27&8&35.4\\
\hline
dc1&$-$&$-$&$-$&$-$&$-$&$-$&1.19&1234&2&0.29&303&80.6\\
\hline
dc2&$-$&$-$&$-$&$-$&$-$&$-$&1.20&1206&2&0.51&66&54.9\\
\hline
dc3&$-$&$-$&$-$&$-$&$-$&$-$&1.22&1311&6&0.33&72&49.1\\
\hline
trans4&$-$&$-$&$-$&$-$&$-$&$-$&2.30&1984&0&0.23&13&8.51\\
\hline
trans5&$-$&$-$&$-$&$-$&$-$&$-$&2.14&1681&0&0.22&11&11.2\\
\hline
\end{tabular}
\end{table}

We see that all the $a<1$ for all the test problems. This again confirms the
reliability of our choices of $c_0$, $c_1$ and $c_2$.
We observe from the table that S-RSAI($tol$) fails to compute $M$
for the last six larger problems. This is because each $A$ has some
fully dense rows so that RSAI($tol$) involves solutions of some large LS
problems \eqref{reducemin2norm} with dimensions $p_{dr}\approx n$
at each loop, which exceeds the RAM of our computer.
However, with the same parameters,
N-RSAI($tol$) computes the $M$ for all the double regular sparse matrices very
efficiently and costs comparable CPU time to N-PSAI($tol$) except rajat13.
Meanwhile, these $M$ are effective for preconditioning and have similar
effects to those obtained by N-PSAI($tol$), as shown by the corresponding
$n_c$ and $iter$.  For rajat04, rajat12 and memplus, the $M$
constructed by S-RSAI($tol$) are much denser than those by N-RSAI($tol$), meaning
that S-RSAI($tol$) costs much more CPU time than N-RSAI($tol$) does,
but the $M$ by S-RSAI($tol$) have comparable preconditioning quality to
those by N-RSAI($tol$).
For rajat13, we see that S-RSAI($tol$) fails to construct $M$ within 25 hours.
This is because $r_k(i)$ has one dominant index $i$ that corresponds to
a dense row of $A$ at some loop when computing $m_k$ for $k=1,2,\ldots,n$, causing
totally $n$ large LS problems to emerge,
which is a huge computational task. Actually,
we have found that S-RSAI($tol$) costs about 81 hours to solve this problem.
In contrast, N-RSAI($tol$) is very efficient and uses only 13.6 seconds to
construct $M$, and Algorithm 2 costs only 14 seconds to solve the problem,
a very striking improvement over the standard approach!

Finally, we summarize the six algorithms S-SPAI, S-PSAI($tol$), and
S-RSAI($tol$) and N-SPAI, N-PSAI($tol$), and N-RSAI($tol$). As can be observed from
the numerical experiments, for the standard approach the CPU time $ptime$ of
constructing $M$ dominates the overall efficiency and it overwhelms the CPU
time $stime$ of solving the preconditioned linear systems by BiCGstab. This is
in accordance with the remark paragraph in the end of Section~\ref{sec:3}.
As a matter of fact, the numerical experiments have indicated
that, due to the memory storage and huge computational cost,
SPAI, PSAI($tol$) and RSAI($tol$) cannot generate
preconditioners when directly applied to six of the ten test irregular sparse
problems. By contrast, the transformation approach is very successful
and solves all the ten test double irregular sparse problems very efficiently.

\section{Conclusions}\label{sec:6}

We have considered SPAI, PSAI($tol$) and RSAI($tol$) for
double irregular sparse linear systems, which are common in practical applications.
We have shown that they are costly and even impractical for this class of problems,
but much cheaper to construct effective preconditioners for double regular
sparse ones. To fully exploit these three preconditioning procedures,
by making use of the Sherman-Morrison-Woodbury formula, we have proposed an
approach that transforms a double irregular sparse problem into
a small number of double irregular sparse ones with the same coefficient
matrix $\hat{A}$, for which we can use the three procedures to
efficiently construct good preconditioners for the resulting double
regular sparse problems and solve
the preconditioned linear systems by Krylov solvers. We have
considered the non-singularity and conditioning of $\hat{A}$.
To develop the transformation approach into a practical and
reliable algorithm, we have given a number of theoretical and practical
considerations.

The numerical experiments have demonstrated that the
proposed transformation approach has the very sharp superiority to
the standard approach that first preconditions the original double irregular
sparse problem \eqref{equation} by SPAI, PSAI($tol$) or RSAI($tol$) and then
solves the preconditioned linear system by Krylov solvers. The experiments
have also illustrated that the transformation approach greatly improves the
efficiency of the algorithm proposed in \cite{Jia13b} that only
transforms a double irregular sparse problem into
column regular sparse ones with the same coefficient
matrix. We have seen that, due to the memory storage and huge
computational cost, SPAI, PSAI($tol$) and RSAI($tol$) cannot generate
preconditioners when directly applied to six of the ten real-world
double irregular sparse problems.

\end{document}